\documentclass[12pt]{amsart}

\usepackage[margin=1.15in]{geometry}
\usepackage{amsmath,amscd,amssymb,amsfonts,latexsym,wasysym, mathrsfs, mathtools,hhline,color, bm}
\usepackage[all, cmtip]{xy}

\usepackage{url}

\definecolor{hot}{RGB}{65,105,225}

\usepackage[pagebackref=true,colorlinks=true, linkcolor=hot ,  citecolor=hot, urlcolor=hot]{hyperref}

\usepackage{ textcomp }
\usepackage{ tipa }

\usepackage{graphicx,enumerate}

\newcommand{\CN}{\mathbb{C}^{n}}
\newcommand{\CC}{\mathbb{C}}

\newcommand{\VV}{\mathbf{V}}

\newcommand{\bfx}{\mathbf{x}}
\newcommand{\bfu}{\mathbf{u}}
\newcommand{\bfp}{\mathbf{p}}
\newcommand{\bfy}{\mathbf{y}}
\newcommand{\bfz}{\mathbf{z}}

\theoremstyle{plain}
\newtheorem{theorem}{Theorem}[section]
\newtheorem{proposition}[theorem]{Proposition}

\newtheorem{lm}[theorem]{Lemma}

\newtheorem{corollary}[theorem]{Corollary}

\newtheorem{lemma}[theorem]{Lemma}
\newtheorem{thrm}[theorem]{Theorem}

\theoremstyle{definition}

\newtheorem{remark}[theorem]{Remark}

\newtheorem{ex}[theorem]{Example}
\newtheorem*{ex*}{Example}

\def\be{\begin{equation}}
\def\ee{\end{equation}}

\def\bt{\begin{thrm}}
\def\et{\end{thrm}}

\def\bc{\begin{cor}}
\def\ec{\end{cor}}

\def\br{\begin{rmk}}
\def\er{\end{rmk}}

\def\bp{\begin{prop}}
\def\ep{\end{prop}}

\def\bl{\begin{lm}}
\def\el{\end{lm}}

\def\bex{\begin{ex}}
\def\eex{\end{ex}}

\def\bd{\begin{defn}}
\def\ed{\end{defn}}

\def\bp{\mathbf{p}}

\newcommand\hh{{\mathfrak{h}}}

\DeclareMathOperator{\reg}{reg}                  
                  
\DeclareMathOperator{\id}{id}                    

\DeclareMathOperator{\MLdeg}{MLdeg}

\DeclareMathOperator{\LOdeg}{LOdeg}

\def\bC{\mathbb{C}}

\def\bP{\mathbb{P}}
\def\cH{\mathcal{H}}

\def\cO{\mathcal{O}}
\def\lra{\longrightarrow}

\title[Linear optimization on varieties and Chern-Mather classes]{Linear optimization on varieties and \\Chern-Mather classes}

\author{Laurentiu G. Maxim}
\address{Department of Mathematics,         University of Wisconsin-Madison,  480 Lincoln Drive, Madison WI 53706-1388, USA.}
\email {maxim@math.wisc.edu}\urladdr{https://www.math.wisc.edu/~maxim/}
\author{Jose Israel Rodriguez}
\address{Department of Mathematics,         University of Wisconsin-Madison,  480 Lincoln Drive, Madison WI 53706-1388, USA.}
\email {jose@math.wisc.edu}\urladdr{https://sites.google.com/wisc.edu/jose/}
\author{Botong Wang}
\address{Department of Mathematics,         University of Wisconsin-Madison,  480 Lincoln Drive, Madison WI 53706-1388, USA.}
\email {wang@math.wisc.edu}\urladdr{http://www.math.wisc.edu/~wang/}

\author{Lei Wu}
\address{Department of Mathematics, KU Leuven, Celestijnenlaan 200B B-3001 Leuven, Belgium}
\email {lei.wu@kuleuven.be}\urladdr{https://sites.google.com/view/leiwuswebsite/}



\begin{document}

\date{\today}

\maketitle

\begin{abstract}  
The linear optimization degree gives an algebraic measure of complexity of optimizing a linear objective function over an algebraic model. 
Geometrically, it can be interpreted as the degree of a projection map on the {affine} conormal variety. 
Fixing an affine variety, our first result shows that the geometry of {this} conormal variety, expressed in terms of bidegrees, completely determines the Chern-Mather classes of the given variety. We also show that these bidegrees coincide with the linear optimization degrees of generic affine sections. 
\end{abstract}


\section{Introduction}\label{intro}
For a complex projective variety $X\subset \bP^n$, the \emph{maximum likelihood (ML) degree} of $X$, denoted by $\MLdeg(X)$, is defined to be the number of critical points of a general likelihood function ${p_0^{u_0}\cdots p_n^{u_n}}/{(p_0+\cdots +p_n)^{u_0+\cdots +u_n}}$, with $u_i\in \mathbb{Z}$, on the smooth locus of $X\setminus \mathcal{H}$, where $\mathcal{H}$ is the union of all coordinate hyperplanes and the hyperplane {given by} $p_0+\cdots +p_n=0$. 
When $X\setminus \cH$ is smooth, $\MLdeg(X)$ is equal, up to a sign, to the Euler characteristic of $X\setminus \cH$ (see \cite{Huh}). When $X\setminus \cH$ is singular, $\MLdeg(X)$ is equal to the Euler characteristic of MacPherson's local Euler obstruction function $Eu_{X\setminus \cH}$ (see \cite{RW} and  \cite{MRWW}). Noting that the Euler characteristic is the degree of the total Chern class, the above results can be extended to relations between the ML bidegrees and MacPherson's Chern and Chern-Mather  classes.
 Moreover, using a Chern class/Euler characteristic involution formula of Aluffi, relations between ML bidegrees and sectional ML degrees are established  in \cite{HS} and \cite{MRWW}. In particular, in their recent paper \cite{MRWW}, the authors proved the Huh-Sturmfels \emph{involution conjecture} of \cite{HS}.

In this paper, we aim to find a linear analogue of the above-mentioned results. Given an {\it affine} variety $X\subset \bC^n$, we define its {\it linear optimization (LO) degree}, denoted by $\LOdeg(X)$, to be the number of critical points of a general linear function restricted to the smooth locus $X_{\textrm{reg}}$ of $X$. 
This gives an algebraic measure to the complexity of optimizing a linear function over algebraic models $X_{\reg} \cap \mathbb{R}^n$,
which are prevalent in algebraic statistics and applied algebraic geometry. 
Similar to the ML degrees, we can also define LO bidegrees $b_i(X)$ and sectional LO degrees $s_i(X)$, as we will discuss below. 
Our first result (Theorem \ref{thm_main}) is to relate the LO bidegrees $b_i(X)$ with the Chern-Mather class of $X$. 
Furthermore, it is the case that $s_i(X)\leq b_i(X)$, see Section~\ref{sec:applied}, 
and our second result (Theorem \ref{thm_bs}) states that the equality always holds. 

An equivalent definition of the linear optimization degree $\LOdeg(X)$ of an affine variety $X\subset \bC^n$ can be given as follows.   Let $T^*_X\bC^n$ be  the {affine} conormal variety of $X$, i.e., the closure of the conormal bundle $T^*_{X_{\textrm{reg}}}\bC^n$ of $X_{\textrm{reg}}$ in $T^*\bC^n$. Consider the trivialization $T^*\bC^n\cong \bC^n\times \bC^n$ of the cotangent bundle, where the first factor is the base and the second is the fiber. Then the projection of $T^*_X\bC^n$ to the second factor $\bC^n$ is a generically finite map, and its degree is equal to $\LOdeg(X)$.

We define the LO bidegrees of $X$ to be the bidegrees of $T^*_X\bC^n$. More precisely, consider the standard compactification $\bC^n\times \bC^n\subset \bP^n\times \bP^n$, 
and let $\overline{T^*_X\bC^n}$ be the closure of $T^*_X\bC^n$ in $\bP^n\times \bP^n$. We define the {\it LO bidegrees} of $X$, denoted by $b_i(X)$ or simply $b_i$, to be the coefficients of the Chow class of $\overline{T^*_X\bC^n}$, that is,
\be\label{bi}
[\overline{T^*_X\bC^n}]=b_0 [\bP^0\times \bP^n]+b_1[\bP^{1}\times \bP^{n-1}]+\cdots +b_d[\bP^{d}\times \bP^{n-d}]\in A_*(\bP^n\times \bP^n)
\ee
where $d=\dim X$. In particular, $b_0(X)=\LOdeg(X)$.

Fixing the standard compactification $\bC^n\subset \bP^n$, we consider the local Euler obstruction function $Eu_X$ of the affine variety $X \subset \bC^n$ as a constructible function on $\bP^n$, with value 0 outside of $X$. Applying to it  the Chern-MacPherson transformation $c_*:F(\bP^n) \to A_*(\bP^n)$, with $F(\bP^n)$ the group of constructible functions on $\bP^n$, we get a class
\be\label{cma}
c^{Ma}(X):=c_*(Eu_X)=a_0[\bP^0]+a_1 [\bP^1]+\cdots+a_d [\bP^d]\in A_*(\bP^n),
\ee
which we refer to as 
the \emph{total Chern-Mather class of $X$}. To emphasize the space $X$ we work with,  we will occasionally use the notation $a_i(X)$ for the coefficients $a_i$ of \eqref{cma}.

For notational convenience, in \eqref{bi} and \eqref{cma} we set $a_j=b_j=0$ if $j\notin \{0, 1, \ldots, d\}$.

Our first result describes the relation between the LO bidegrees and the total Chern-Mather class of $X$ as follows.
\begin{theorem}\label{thm_main}
For any $d$-dimensional irreducible affine variety $X\subset \bC^n$, the sequences $\{a_i\}$ and $\{b_i\}$ defined as in \eqref{bi} and \eqref{cma} satisfy the identity
\be\label{id1}
\sum_{0\leq i\leq d}b_i t^{n-i}= \sum_{0\leq i\leq d}a_i (-1)^{d-i} t^{n-i}(1+t)^{i}.
\ee
\end{theorem}

Let us state two immediate consequences of Theorem \ref{thm_main}, which were also considered by other authors by different methods.

First, the equality of top degree coefficients in \eqref{id1} reproves the following result of Seade-Tib\u{a}r-Verjovsky \cite[Equation (2)]{STV} (see also \cite[Theorem 1.2]{ST} and \cite[Theorem 3.10]{MRW}).

\begin{corollary}\label{cor_linear}
For any $d$-dimensional irreducible affine variety $X\subset \bC^n$, and $H\subset \bC^n$ a general affine hyperplane, we have
\begin{equation}\label{eq_b0}
\LOdeg(X)=b_0(X)=(-1)^{d} \cdot \chi(Eu_{X}|_{\bC^n\setminus H}).
\end{equation}
\end{corollary}

Secondly, by plugging $t=-1$ in \eqref{id1}, we derive the following relation between the value of the local Euler obstruction function of an affine cone at the cone point, and the LO bidegrees of the affine cone. More precisely, in the notations of \eqref{bi}, we get the following result (compare also with~\cite[Corollaire 5.1.2]{polaires}).
\begin{corollary}\label{cor_cone}
Assume that the $d$-dimensional irreducible affine variety $X\subset \bC^n$ 
is an affine cone of a projective variety, and denote its cone point by $O$. Then
\be\label{euc}
Eu_X(O)=b_d(X)-b_{d-1}(X)+\cdots +(-1)^db_0(X).
\ee
\end{corollary}

\medskip

By analogy with the sectional maximum likelihood degrees, we now introduce sectional LO degrees of affine varieties as follows. For any $0\leq i\leq d$, we define the {\it $i$-th sectional  LO degree} of $X$, denoted by $s_i(X)$ or simply $s_i$, to be 
\be\label{si} s_i(X):= \LOdeg(X\cap H_1\cap \cdots \cap H_i),\ee
where $H_1, \ldots, H_i$ are generic affine hyperplanes. Then $s_0(X)=\LOdeg(X)$, and $s_d(X)$ is the degree of $X$. Here, for notational convenience, we also set $s_i=0$ for $i>d$. 

Our next result shows that the LO bidegrees and sectional LO degrees coincide.
\begin{theorem}\label{thm_bs}
Let $X\subset \bC^n$ be any irreducible affine variety, and let $b_i$ and $s_i$ be its LO bidegrees and LO sectional degrees, respectively. Then $s_i=b_i$ for all $i$. 
\end{theorem}

Our formula in Theorem \ref{thm_main} shows that the Chern-Mather class of the affine variety $X$ is determined by the LO bidegrees. The relationship is more involved than the corresponding result for ML bidegrees (\cite[Theorem 1.3]{MRWW}) because, while the logarithmic cotangent bundle of the pair $(\bP^n, \bP^n\setminus (\bC^*)^n)$ is trivial,  the one of $(\bP^n, \bP^n\setminus \bC^n)$ is not. (See Proposition \ref{prop_trivial} for a remedy of this issue.) By contrast, Theorem \ref{thm_bs} shows that there is a simple relation between the LO bidegrees and the sectional LO bidegrees, unlike the ML degree situation where the relationship is given by an involution formula (see \cite[Theorem 1.5]{MRWW}). 
In particular our result gives, via~\eqref{eq_b0} and \eqref{si}, a topological interpretation of all LO bidegrees as
Euler characteristics, that is,
\[
b_i(X)=(-1)^{d-i}\chi(Eu_{X\cap H_1\cap \cdots \cap H_i} |_{\bC^n\setminus H_{i+1}}),
\]
with $d=\dim X$. 
The equality between LO bidegrees and the sectional LO degrees also shows that, when computing the LO bidegrees, orthogonal subspaces are sufficiently general (see Corollary~\ref{cor:generic}). 

\medskip

In Section \ref{sec_compare}, we discuss the relation between the LO bidegrees of an affine variety and the polar degrees of its projective closure (see Proposition \ref{prop_iff}). As a consequence, we generalize Theorem~13 of \cite{wasserstein} to singular varieties (see Corollary \ref{cor_si}). 
In view of formula \eqref{euc}, this relation also allows us to express the value of local Euler obstruction of an affine cone at the cone point in terms of the projective polar degrees of the projective variety we are coning off (compare with \cite[Proposition~3.17]{Aluffi18}).

\begin{remark}
We believe our results here motivate further analogous investigations for other objective functions like Euclidean distance~\cite{MR3451425}, $p$-norms~\cite{kubjas2021algebraic} and bottlenecks~\cite{bottleneck-degree}.
Moreover, Proposition~\ref{prop:LM} encourages a revisitation into the maximum likelihood estimation case~\cite{Huh,MR2230921, MR2189544} to find an involution~ at the level of critical points. 
\end{remark}

\subsection{Comparison with other works}
Let us clarify here the difference between our approach and some of the more classical works. 

As above, let $X \subset \bC^n$ be an irreducible affine variety with conormal space $T^*_X\bC^n \subset T^*\bC^n$. Instead of taking the fiberwise projectivization $C(X,\bC^n):=\bP(T^*_X\bC^n) \subset \bP(T^*\bC^n)$ as in, e.g., Sabbah \cite{Sab}, we first compactify the fibers of $T^* \bC^n$ by taking their projective closures, 
i.e., $T^*\bC^n=\bC^n\times \bC^n\subset \bC^n\times \bP^n$, 
so that we keep track of conic subvarieties contained in the zero section of $T^*\bC^n$, and then we compactify $\bC^n\times \bP^n$ using the trivial projective bundle $\bC^n\times \bP^n\subset \bP^n\times \bP^n$. Other authors, like Aluffi \cite{Aluffi18} or Parusinski-Pragacz \cite{PP01}, consider the projective closure $\overline{X} \subset \bP^n$ of $X$, together with its corresponding projective conormal variety $C(\overline{X},\bP^n):=\bP(T^*_{\overline{X}}\bP^n) \subset \bP(T^*\bP^n)$. 

Note that Sabbah's formula \cite[Lemme 1.2.1]{Sab} applied to $X \subset \bC^n$ computes the Chern-Mater class of $X$ in the Borel-Moore homology of $X$. The same formula applied to $\overline{X} \subset \bP^n$ computes the Chern-Mather class of $\overline{X}$ in the Borel-Moore homology (or Chow group) of $\overline{X}$, and resp., of  $\bP^n$, upon using the proper pushforward. By contrast, 
we relate our compactification of $T^*\bC^n$ in $\bP^n\times \bP^n$ to a twisted logarithmic cotangent bundle of $\bP^n$, and compute the Chern-Mather class of $X$ in $A_*(\bP^n)$ via Ginsburg's microlocal interpretation of Chern-MacPherson classes (cf. \cite{Gin}).
In fact,  we derive Theorem~\ref{thm_main} as a consequence of our main result from \cite[Theorem 1.1]{MRWW}, recalled below in Theorem~\ref{thm_m},
which computes the Chern classes of the extension by zero to 
$\mathbb{P}^n$ of the local Euler obstruction function $Eu_X$ of the affine variety $X\subset \bC^n$.

This kind of relation between conormal varieties, Chern classes, and polar varieties, has been already considered by \cite{Sab}, \cite{polaires},  \cite{Aluffi18},  etc.
For example, when $X$ is the affine cone on a projective variety,
or more generally, if the projective closure  $\overline{X}$ of $X$ is  transversal to the hyperplane at infinity $H_\infty$ of $\mathbb{P}^n$,
 Theorem~\ref{thm_main} can be derived from a combination of results contained in \cite{Sab} and \cite{Aluffi18}. This is the case when $H_\infty$ is not contained in the dual variety of $\overline{X}$, see Section~\ref{sec_compare} for more results in this direction. 
 
 The novel contribution of Theorem~\ref{thm_main} (and of its consequence in Theorem \ref{thm_bs})
is that it applies to all affine varieties without any additional assumption of infinity.  
For example, in 
\cite{MRW} we prove a conjecture from \cite{MR3451425} by applying formula \eqref{eq_b0} to the computation of the Euclidean distance degree of the multiview variety, which does not have good behavior along infinity.

\medskip

{\bf Acknowledgements.} 
The authors thank Bernd Sturmfels for inspiring comments on an earlier version of the paper. 
Maxim  is  partially  supported  by  the  Simons  Foundation  (Collaboration  Grant  \#567077),  and  by  the  Romanian  Ministry  of  National  Education (CNCS-UEFISCDI grant PN-III-P4-ID-PCE-2020-0029).  Rodriguez is partially supported by the Office of the Vice Chancellor for Research and Graduate Education at UW-Madison with funding from the Wisconsin Alumni Research Foundation.
 Wang is partially supported by a Sloan fellowship.  Wu is supported by an FWO postdoctoral fellowship.


\section{Characteristic cycles. Chern classes. Microlocal interpretation}

In this paper, we work in the complex algebraic context, with $A_*$ denoting the Chow group. By convention, we use subscripts for characteristic classes valued in Chow groups, and we use superscripts whenever a characteristic class is of cohomological nature (e.g., Chern classes of a vector bundle).

Let $X$ be a smooth complex algebraic variety, and denote by $F(X)$ the group of algebraically constructible functions on $X$, i.e., the free abelian group generated by indicator functions $1_Z$ of closed irreducible subvarieties $Z$ of $X$. An important example of a constructible function on $X$ is the MacPherson {\it local Euler obstruction} function $Eu_Z$ of an irreducible subvariety $Z$ of $X$, see  \cite{MP0}. 

Let $L(X)$ be the free abelian group spanned by the irreducible conic Lagrangian cycles in the cotangent bundle $T^*X$. Recall that irreducible conic Lagrangian cycles in $T^*X$ correspond to the conormal spaces $T^*_ZX$, for $Z$ a closed irreducible subvariety of $X$.  
Here, for such a closed irreducible subvariety  $Z$ of $X$ with smooth locus $Z_{\reg}$, its conormal variety $T^*_{Z}X$ is defined as the closure in $T^*X$ of
\[
T^*_{Z_{\reg}}X:=\{
(z,\xi)\in T^*X \mid z \in Z_{\reg}, \ \xi \in T^*_zX, \ \xi\vert_{T_zZ_{\rm reg}}=0 \}.
\]

The characteristic cycle functor $CC$ establishes a group isomorphism
$$CC:F(X) \lra L(X),$$ 
which, for a closed irreducible subvariety $Z$ of $X$, satisfies:
\be\label{cc} CC(Eu_Z)=(-1)^{\dim Z}\cdot T^*_Z X.\ee

In \cite{MP0}, MacPherson extended the notion of Chern classes to singular complex algebraic varieties by defining a natural transformation
$$c_*:F(-) \lra A_*(-)$$
from the functor $F(-)$ of constructible functions (with proper morphisms) to Chow (or Borel-Moore) homology, such that if $X$ is a smooth variety then $c_*(1_X)=c^*(TX) \cap [X]$. Here,  $c^*(TX)$ denotes the total cohomology Chern class of the tangent bundle $TX$, and $[X]$ is the fundamental class of $X$.
For any 
locally closed irreducible subvariety
$Z$ of a complex algebraic variety $X$, the function $1_Z$ is constructible on $X$, and the class 
$$c_*^{SM}(Z):=c_*(1_Z) \in A_*(X)$$
is usually referred to as the {\it Chern-Schwartz-MacPherson (CSM) class} of $Z$ in $X$. Similarly, 
the class $$c_*^{Ma}(Z):=c_*(Eu_Z) \in A_*(X)$$ is called the {\it Chern-Mather class} of $Z$, where we regard the local Euler obstruction function $Eu_Z$ as a constructible function on $X$ by setting the value zero on $X \setminus Z$.

Results of Ginsburg \cite{Gin} and Sabbah \cite{Sab} provided a microlocal interpretation of Chern classes, by showing that McPherson's Chern class transformation $c_*$ factors through the group of conic Lagrangian cycles in the cotangent bundle. We recall this construction below, following, e.g., \cite{AMSS}.

Let $E$ be a rank $r$ vector bundle on the smooth complex algebraic variety $X$. Let $\overline{E}\coloneqq \bP(E\oplus \mathbf{1})$ be the projective bundle, which is a fiberwise compactification of $E$ (with $\mathbf{1}$ denoting the trivial line bundle on $X$). Then $E$ may be identified with the open complement of $\bP(E)$ in $\overline{E}$. Let $\pi:E\to X$ and $\bar{\pi}:\overline{E} \to X$ be 
the projections, and let $\xi:=c^1(\mathcal{O}_{\overline{E}}(1))$
 be the first Chern class of the hyperplane line bundle on $\overline{E}$.
Pullback via ${\bar \pi}$ realizes $A_*(\overline{E})$ as a $A_*(X)$-module. 
An irreducible conic $d_C$-dimensional subvariety $C \subset E$
determines a ${d_C}$-dimensional cycle $\overline{C}$ in $\overline{E}$ and, by \cite[Theorem 3.3]{Ful}, one can express $[\overline{C}] \in A_{d_C}(\overline{E})$ uniquely as:
\begin{equation}\label{sh}
    [\overline{C}]=\sum_{j={d_C}-r}^{d_C} \xi^{j-{d_C}+r} \cap {\bar \pi}^* c^E_j(C),
\end{equation}
for some $c^E_j(C) \in A_{j}(X)$. The classes $$c_{{d_C}-r}^E(C), \ldots, c_{d_C}^E(C)$$ defined by \eqref{sh} are called the {\it Chern classes of $C$}. The sum \[ c_*^E(C)=\sum_{j={d_C}-r}^{d_C} c^E_j(C)\] is called the {\it shadow} of $[\overline{C}]$. For our applications, we will mainly work with conic Lagrangian cycles in cotangent bundles, in which case we have ${d_C}=r$. In fact, the terminology ``Chern classes of $C$"
 is justified by the following result, applied to the cotangent bundle $T^*X$ and elements of the group $L(X)$ of conic Lagrangian cycles:

\begin{proposition}\cite[Proposition~3.3]{AMSS}\label{pmc}
For any constructible function $\varphi \in F(X)$, the Chern classes of the characteristic cycle $CC(\varphi)$ equal the signed MacPherson Chern classes of $\varphi$, namely:
\begin{equation}\label{mc}
    c_j^{T^*X}\left(CC(\varphi)\right)=(-1)^j \cdot  c_j(\varphi) \in A_j(X), \ \ j=0,\ldots, \dim(X),
\end{equation}
where  $c_j(\varphi)$ denotes the $j$-th component of MacPherson's  Chern class $c_*(\varphi)$.
\end{proposition}

If $Z\subset X$ is a closed irreducible subvariety, one gets from \eqref{cc} and \eqref{mc} the following identity:
\be\label{tc}
c_*^{T^*X}(T_{Z}^*X)=(-1)^{\dim Z}\sum_{j\geq 0}(-1)^j{c}^{Ma}_{j}(Z),
\ee
with ${c}^{Ma}_{j}(Z)$ denoting the $j$-th component of the Chern-Mather class of $Z$.

\medskip

We end this section by recalling our main result from \cite{MRWW}, which was used there for proving the Huh-Sturmfels involution conjecture in maximum likelihood estimation.

Let $X$ be a smooth complex algebraic variety, and let $D\subset X$ be a normal crossing divisor. Let $U:=X\setminus D$ be the complement, and let $j:U\hookrightarrow X$ be the open  inclusion. Let $\Omega_X^1(\log D)$ be the sheaf of algebraic one-forms with logarithmic poles along $D$, and denote the total space of the corresponding vector bundle by $T^*(X, D)$. Note that $T^*(X, D)$ contains $T^*U$ as an open subset. Given a conic Lagrangian cycle $\Lambda$ in $T^*U$, we denote its closure in $T^*(X, D)$ by $\overline\Lambda_{\log}$. With these notations, the following result~holds.

\begin{theorem}\label{thm_m}\cite[Theorem 1.1]{MRWW} 
Let $\varphi \in F(U)$ be any constructible function on $U$. Then 
\be\label{eq_main}
c^{T^*(X, D)}_*\Big(\overline{CC(\varphi)}_{\log}\Big)=c^{T^*X}_*\big(CC(\varphi)\big) \in A_*(X),
\ee
where, if $CC(\varphi)=\sum_{k}n_k\Lambda_k$, then $\overline{CC(\varphi)}_{\log}\coloneqq\sum_{k}n_k(\overline{\Lambda_k})_{\log}$. Here, on the right-hand side of \eqref{eq_main}, $\varphi$ is regarded as a constructible function on $X$ by extension by zero.
\end{theorem}

In particular, if $\varphi = Eu_Z$ for $Z \subset U$ an irreducible subvariety, then for $\Lambda=T_{Z}^*U$ we get from \eqref{tc} and \eqref{eq_main} that:
\be\label{8}
c^{T^*(X, D)}_*(\overline\Lambda_{\log})=(-1)^{\dim Z}\sum_{j\geq 0}(-1)^j{c}^{Ma}_{j}(Z) \in A_*(X).
\ee

Formula \eqref{8} will play a fundamental role in the proof of Theorem~\ref{thm_main} in Section~\ref{s:proofs}.


\section{Segre classes and Shadow of twisted cycles}
Let $C$ be a cone over a variety $Y$, typically a subcone of a vector bundle. Let $\bP(C)$ be the projectivization of $C$, with projection $\pi:\bP(C) \to Y$. We also let $\bP(C \oplus \bf{1})$ be the projective completion of $C$, with projection map $\overline{\pi}$. Denote the  tautological line bundle on $\bP(C \oplus \bf{1})$ by $\mathcal{O}_{\bP(C \oplus \bf{1})}(-1)$, and denote its inverse by $\mathcal{O}_{\bP(C \oplus \bf{1})}(1)$. Following \cite[Chapter 4]{Ful}, we define the {\it Segre class} of $C$, denoted $s_*(C)$, to be the class in $A_*(Y)$ defined by:
\be
s_*(C):=\overline{\pi}_* \left(\sum_{i \geq 0} c^1(\mathcal{O}_{\bP(C \oplus \bf{1})}(1))^i \cap  [\bP(C \oplus \bf{1})] \right).
\ee
The $i$-th Segre class $s_{i}(C)$ is the $i$-th graded piece of $s_*(C)$. If the cone $C$ is of pure dimension $d_C$ over $Y$, then:
\[ s_{i}(C)
=
\overline{\pi}_{*}\left(c^1({\mathcal{O}_{\bP(C \oplus \bf{1})}}(1))^{d_C-i} \cap [\mathbb {P} (C\oplus {\bf 1})]\right) \in  A_i(Y).\]
\begin{ex}
If $E$ is a vector bundle on $Y$, then $s_*(E)=c^*(E)^{-1} \cap [Y]$; see \cite[Proposition 4.1(a)]{Ful}.
\end{ex}
\begin{remark}\label{eqf}
The addition of the trivial factor ${\bf 1}$ is needed to account for the possibility that $\bP(C)$ may be empty, e.g., when $C$ is contained in the zero section of a vector bundle. However, if $C$ is an irreducible conic variety such that $\bP(C)$ is nonempty, then cf. \cite[Example 4.1.2]{Ful}, we have:
\be\label{noe}
s_*(C):={\pi}_* \left(\sum_{i \geq 0} c^1(\mathcal{O}_{\bP(C)}(1))^i \cap  [\bP(C)] \right).
\ee
In particular, in this case, we have $s_i(C)=0$ for $i\geq \dim C$.
\end{remark}

Let $Y$ be a smooth projective variety
and let $D$ be a reduced
divisor with complement $U\coloneqq Y\setminus D$. Let $E$ be a vector bundle on $Y$, and let $C\subset E$ be an irreducible conic subvariety whose support in $Y$ is not contained in $D$. We consider $E|_U$ as the common open subset of $E$ and $E(D)\coloneqq E\otimes \mathcal{O}_Y(D)$. Denote the closure of $C\cap (E|_U)$ in $E(D)$ by $C'$. The following proposition is a straightforward generalization of 
\cite[Example 3.1.1]{Ful}.
\begin{proposition}\label{prop_Segre}
Under the above notations, we denote the dimension of $C$ by $d_C$. If $C$ is not contained in the zero section of $E$, then the Segre classes of $C$ and $C'$ are related by the identity 
\begin{equation}\label{eq_Segre}
s_{d_C-i-1}(C')=\sum_{0\leq j\leq i} {i \choose j}(-[D])^{i-j}\cap s_{d_C-j-1}(C) \quad\text{for all $i\geq 0$}.
\end{equation}
\end{proposition}

\begin{proof}
Under the natural isomorphism $\bP(E)=\bP(E(D))$, the projectivization $\bP(C)$ of $C$ is the same as that of $C'$. 
Notice that 
$\mathcal{O}_{\bP(E(D))}(-1)=\mathcal{O}_{\bP(E)}(-1)\otimes \pi^*(\mathcal{O}_Y(D))$.
Moreover, the pullback of $c^1(\mathcal{O}_{\bP(E)}{(1)})$ to $\bP(C)$ is equal to $c^1(\mathcal{O}_{\bP(C)}(1))$. 
Thus, the Segre classes of $C$ and $C'$ can also be expressed as
\[
s_*(C)=\pi_{*}\left(\sum_{i\geq 0} c^1\left(\mathcal{O}_{\bP(E)}(1)\right)^i \cap [\bP(C)]\right)
\]
and 
\[s_*(C')=\pi_{*}\left(\sum_{i\geq 0} \left(c^1(\mathcal{O}_{\bP(E)}(1))-\pi^*[D]\right)^i \cap [\bP(C)]\right)
\]
where 
$\pi: \bP(E)=\bP(E(D))\to Y$ is the projective bundle map. 

Combining the above equations and using the projection formula, we have 
\begin{align*}
s_{d_C-i-1}(C')&=\pi_{*}\left(\left(c^1(\mathcal{O}_{\bP(E)}(1))-\pi^*[D]\right)^i \cap [\bP(C)]\right)\\
&=\sum_{0\leq j\leq i} {i \choose j}(-[D])^{i-j}\cap \pi_{*}\left(c^1(\mathcal{O}_{\bP(E)}(1))^j \cap [\bP(C)]\right)\\
&=\sum_{0\leq j\leq i} {i \choose j}(-[D])^{i-j}\cap s_{d_C-j-1}(C)
\end{align*}
for any $i\geq 0$. 
\end{proof}

We can use the following elementary formula to simplify formula \eqref{eq_Segre}. 
\begin{lemma}\label{lemma_binomial}
As power series, we have the following identity:
\[
\sum_{k\geq 0}{k+n \choose n}(-t)^k=\left( 1-t+t^2-\cdots \right)^{n+1}=(1+t)^{-n-1}.
\]
\end{lemma}

\begin{corollary}\label{c1}
Under the notations and assumptions of Proposition \ref{prop_Segre}, we have
\begin{equation}\label{eq_D}
s_{*}(C')=\sum_{j\geq 0} (1+[D])^{-j-1} \cap s_{d_C-j-1}(C).
\end{equation}
\end{corollary}
\begin{proof}
By Proposition \ref{prop_Segre} and Lemma \ref{lemma_binomial}, we have
\begin{align*}\label{eq_Segre}
s_{*}(C')&=\sum_{j, k\geq 0}{k+j \choose j}(-[D])^{k}\cap s_{d_C-1-j}(C)\\
&=\sum_{j\geq 0}\left(\sum_{k\geq 0}{k+j \choose j}(-[D])^{k} \right)\cap s_{d_C-1-j}(C)\\
&=\sum_{j\geq 0}(1+[D])^{-j-1}\cap s_{d_C-1-j}(C). \qedhere
\end{align*}
\end{proof}

\begin{remark}\label{remark_zero}
When the irreducible subvariety $C\subset E$ is contained in the zero section, by definition, we can identify $C$ and $C'$ as the same subvariety of $Y$. Thus, in this case, we have $s_*(C)=s_*(C')$. Moreover, by definition, all Segre classes of $C$ and $C'$ vanish except in degree $d_C$. In other words, $s_*(C)=s_*(C')=s_{d_C}(C)=s_{d_C}(C')$.
\end{remark}

We recall here the following useful fact.
\begin{proposition}\cite[Lemma 2.12]{Alu2}\label{p1}
For any conic subvariety $C$ in a vector bundle $E$ over $Y$, one has
\begin{equation}\label{eq_CS}
c^E_*(C)=c^*(E)\cap s_*(C),
\end{equation}
with $c^*(E)$ denoting the total cohomology Chern class of $E$, and $c^E_*(C)$ the shadow of $C$ (as defined in the previous section).
\end{proposition}
Combining Corollary \ref{c1} with Proposition \ref{p1}, we have the following.
\begin{corollary}\label{cor_TBD}
For any irreducible conic subvariety $C$ in a vector bundle $E$ over $Y$, and $C'$ defined as above, we have  
\begin{equation}\label{eq_cor}
c_*^{E(D)}(C')=\sum_{k\geq 0}(1+[D])^{r-k}\cap c^E_{d_C-k}(C),
\end{equation}
where $r$ is the rank of $E$. 
\end{corollary}
\begin{remark}\label{remark_stop}
When $k>r$, $c^E_{d_C-k}(C)=0$ by \eqref{sh}. So the summation in \eqref{eq_cor} stops at $k=r$.
\end{remark}

\begin{remark}
If $C$ is the zero section of $E$, then the Segre class of $C$ is the fundamental class of $Y$. In this case, by \eqref{eq_CS} we get $c_*(E)\coloneqq c^*(E)\cap [Y]=c^E_*(C)$ and, similarly, $c_*(E(D))=c^{E(D)}_*(C')$. 
{Corollary~\ref{cor_TBD} }
reduces to the well-known (cohomological) Chern class formula \begin{equation}\label{eq_known}
c^*(E(D))=\sum_{0\leq i\leq r}c^{i}(E)\cdot (1+[D])^{r-i}.
\end{equation}
\end{remark}

\begin{proof}[Proof of Corollary \ref{cor_TBD}]
First, we assume that $C$ is not contained in the zero section of~$E$. By equations \eqref{eq_D} and \eqref{eq_CS}, we have
\begin{align*}
c_*^{E(D)}(C')&=c^*(E(D))\cap s_*(C')\\
&=\left(\sum_{i\geq 0}c^i(E)\cdot (1+[D])^{r-i}\right)\cap \left(\sum_{j\geq 0} (1+[D])^{-j-1}\cap s_{d_C-j-1}(C). \right)\\
&=\sum_{i, j\geq 0}\left(c^i(E)\cdot (1+[D])^{r-i-j-1}\right)\cap s_{d_C-j-1}(C) \\
&=\sum_{k\geq 0}(1+[D])^{r-k-1}\cap \left(\sum_{0\leq i\leq k}c^i(E)\cap s_{d_C-k+i-1}(C)\right)\\
&=\sum_{k\geq 0}(1+[D])^{r-k-1}\cap c^E_{d_C-k-1}(C).
\end{align*}
This is equivalent to \eqref{eq_cor}
since, by \eqref{eq_CS} and Remark \ref{eqf}, 
\[
c^E_{d_C}(C)=\sum_{k\geq 0}c^k(E)\cap s_{d_C+k}(C)=0.
\]

When $C$ is contained in the zero section of $E$, by \eqref{eq_CS}, \eqref{eq_known} and Remark~\ref{remark_zero}, we have
\begin{align*}
c_*^{E(D)}(C')&=c^*(E(D))\cap s_{d_C}(C)\\
&=\sum_{0\leq i\leq r}c^{i}(E)\cdot (1+[D])^{r-i}\cap s_{d_C}(C)\\
&=\sum_{0\leq i\leq r}(1+[D])^{r-i}\cap \big(c_{i}(E)\cap s_{d_C}(C)\big)\\
&=\sum_{0\leq i\leq r}(1+[D])^{r-i}\cap c_{d_C-i}^E(C)
\end{align*}
which is the same as \eqref{eq_cor} by Remark \ref{remark_stop}.
\end{proof}


\section{Twisted logarithmic cotangent bundle}\label{sec_twist}
Fix the standard compactification $\bC^n\subset \bP^n$, and denote the complement divisor by $H_\infty$. Denote the coordinate functions of $\bC^n$ by $z_i$, $1\leq i\leq n$. The following proposition will allow us to relate the results in the previous section and the study of LO bidegrees. 
\begin{proposition}\label{prop_trivial}
The twisted logarithmic cotangent bundle $\Omega_{\bP^n}^1(\log H_\infty)(H_\infty)$ is a trivial bundle. Moreover, the 1-forms $dz_i$ extend to global sections of $\Omega_{\bP^n}^1(\log H_\infty)(H_\infty)$, and they form a trivialization of $\Omega_{\bP^n}^1(\log H_\infty)(H_\infty)$.
\end{proposition}
\begin{proof}
Let $p_0, \ldots, p_n$ be the homogeneous coordinates of $\bP^n$ such that $z_i=\frac{p_i}{p_0}$. Let $U_k\subset \bP^n$ be the affine chart defined by $p_k\neq 0$. In $U_0$, the vector bundle $\Omega_{\bP^n}^1(\log H_\infty)(H_\infty)$ is equal to $\Omega^1_{\bC^n}$, and the sections $dz_i$, $1\leq i\leq n$, generate the locally free sheaf $\Omega^1_{\bC^n}$. 

Without loss of generality, we need to show that the sections $dz_i$, $1\leq i\leq n$, extend to sections of $\Omega_{\bP^n}^1(\log H_\infty)(H_\infty)|_{U_1}$ and they generate the locally free sheaf $\Omega_{\bP^n}^1(\log H_\infty)(H_\infty)|_{U_1}$. In fact, in $U_1$, 
\[
dz_i=d\left(\frac{p_i}{p_0}\right)=d\left(\frac{p_i/p_1}{{p_0}/{p_1}}\right)=\frac{d(p_i/p_1)}{p_0/p_1}-\frac{p_i}{p_1}\cdot \frac{d(p_0/p_1)}{p_0/p_1}\cdot \frac{1}{p_0/p_1}.
\]
Clearly, they are sections of $\Omega_{\bP^n}^1(\log H_\infty)(H_\infty)|_{U_1}$. Notice that for $i\geq 2$, 
\[
dz_i=\frac{d(p_i/p_1)}{p_0/p_1} + \frac{p_i}{p_1}\cdot dz_1.
\]
Thus, as subsheaves of $\Omega_{\bP^n}^1(\log H_\infty)(H_\infty)|_{U_1}$,
\[
\cO_{U_1}\cdot (dz_1, \ldots, dz_n)=\cO_{U_1}\cdot \left(\frac{d(p_0/p_1)}{p_0/p_1}\cdot \frac{1}{p_0/p_1}, \frac{d(p_2/p_1)}{p_0/p_1}, \ldots, \frac{d(p_n/p_1)}{p_0/p_1}\right).
\]
Thus, the sections $dz_1, \ldots, dz_n$ generate $\Omega_{\bP^n}^1(\log H_\infty)(H_\infty)|_{U_1}$.
\end{proof}


\section{The proofs}\label{s:proofs}
In this section we prove our main results stated in the introduction.

\begin{proof}[Proof of Theorem \ref{thm_main}]
Recall that $X\subset \bC^n$ is a $d$-dimensional irreducible subvariety. Denote the conormal variety $T^*_X\bC^n$ by $\Lambda$. Let $\overline\Lambda_{\log}$ be the closure of $\Lambda$ in $E\coloneqq \Omega^1_{\bP^n}(\log H_\infty)$. Then, by formula \eqref{8},
\[
c_*^E(\overline\Lambda_{\log})=
(-1)^d \cdot \sum_{j\geq 0}(-1)^{j}{c}^{Ma}_{j}(X).
\]
Therefore, following the notations in Section \ref{intro}, we have
\[
c_*^E(\overline\Lambda_{\log})=(-1)^d \cdot \left(a_0[\bP^0]-a_1 [\bP^1]+\cdots+(-1)^d a_d [\bP^d]\right) \in A_*(\bP^n).
\]
On the other hand, by Proposition \ref{prop_trivial}, $E(H_\infty)=\Omega^1_{\bP^n}(\log H_\infty)(H_{\infty})$ is trivial. 
By formulas \eqref{bi} and \eqref{sh}, we have
\[
c_*^{E(H_\infty)}(\overline\Lambda_{\log}')=b_0[\bP^0]+b_1[\bP^1]+\cdots+b_d[\bP^d]\in A_*(\bP^n),
\]
where $\overline\Lambda_{\log}'$ is the closure of $\Lambda$ in $E(H_\infty)$.
Applying Corollary \ref{cor_TBD} with $Y=\bP^n$, $D=H_\infty$, $C=\overline\Lambda_{\log}$, we obtain the following relations between sequences $a_i$ and $b_i$,
\begin{equation}\label{eq_ab}
\sum_{0\leq i\leq d}b_i t^{n-i}= \sum_{0\leq i\leq d}a_i (-1)^{d-i} t^{n-i}(1+t)^{i},
\end{equation}
as asserted by Theorem \ref{thm_main}.
\end{proof}

\begin{proof}[Proof of Corollary \ref{cor_linear}]
By \cite[Proposition 2.6]{Alu}, 
\[
c_*(Eu_{X}|_H)=\frac{\hh}{1+\hh} \, c_*(Eu_{X})
\]
where $\hh\in A_{n-1}(\bP^n)$ is the hyperplane class. Under the assumption that 
\[
c_*(Eu_X)=a_0[\bP^0]+a_1 [\bP^1]+\cdots+a_d [\bP^d],
\]
we have
\begin{align*}
c_*(Eu_X)-c_*(Eu_{X}|_H)&=\left(1-\frac{\hh}{1+\hh}\right)\left(a_0[\bP^0]+a_1 [\bP^1]+\cdots+a_d [\bP^d]\right)\\
&=\frac{1}{1+\hh}\left(a_0[\bP^0]+a_1 [\bP^1]+\cdots+a_d [\bP^d]\right)\\
&=\sum_{0\leq i\leq d}(-1)^ia_i[\bP^0]+\sum_{1\leq i\leq d}(-1)^{i-1}a_i[\bP^1]+\cdots+a_d[\bP^d].
\end{align*}
Therefore, 
\[
\chi(Eu_{X}|_{\bC^n\setminus H})=\chi(Eu_{X})-\chi(Eu_{X}|_H)=\int_{\bP^n} \left( c_*(Eu_X)-c_*(Eu_{X}|_H) \right)=\sum_{0\leq i\leq d}(-1)^ia_i. 
\]
On the other hand, it follows immediately from \eqref{eq_ab} that $b_0=\sum_{0\leq i\leq d}(-1)^{d-i}a_i$. Therefore,
\[
b_0=(-1)^d\ \cdot \chi(Eu_{X}|_{\bC^n\setminus H}),
\]
as desired. 
\end{proof}

\begin{proof}[Proof of Corollary \ref{cor_cone}]
The degree zero component of the Chern-Mather class $c^{Ma}(X):=c_*(Eu_X) \in A_*(\bP^n)$ equals the Euler characteristic of the local Euler obstruction function. In other words, in the notations of \eqref{cma}, we have \[a_0(X)=\chi(Eu_X).\] Suppose that $X$ is an {affine cone of a (possibly singular) projective variety,} with cone point at the origin $O$. Then $O$ is the only fixed point of the $\bC^*$-action on $X$. Since $Eu_X$ is invariant under the $\bC^*$-action, the Euler characteristic $\chi(Eu_X)$ is equal to $Eu_X(O)$,  the value of $Eu_X$ at the origin $O$.  Thus, we have \[Eu_X(O)=a_0(X).\] Plugging $t=-1$ in \eqref{id1}, we obtain the desired equality \eqref{euc}. 
\end{proof}

\begin{proof}[Proof of Theorem \ref{thm_bs}]
As in the Introduction, we write 
\[
c_*(Eu_X)=a_0[\bP^0]+a_1 [\bP^1]+\cdots+a_d [\bP^d],
\]
where we consider $Eu_X$ as a constructible function on $\bP^n$ which vanishes on $\bP^n\setminus X$. 
By Corollary \ref{cor_linear}, we know that 
\begin{align*}
s_i&=(-1)^{d-i}\chi(Eu_{X\cap H_1\cap \cdots \cap H_i}|_{\bC^n\setminus H_{i+1}})\\
&=(-1)^{d-i}\chi(Eu_{X\cap H_1\cap \cdots \cap H_i})+(-1)^{d-i-1}\chi(Eu_{X\cap H_1\cap \cdots \cap H_{i+1}})
\end{align*}
where $H_1, \ldots, H_{i+1}$ are general affine hyperplanes in $\bC^n$. 
By \cite[Proposition 2.6]{Alu}, 
\[
\chi(Eu_{X\cap H_1\cap \cdots \cap H_i})=\int_{\bP^n}\left(\frac{\hh}{1+\hh}\right)^i c_*(Eu_X),
\]
where $\hh\in A_{n-1}(\bP^n)$ is the hyperplane class. Therefore, 
\begin{align*}
s_i&=(-1)^d\int_{\bP^n}\left(\left(-\frac{\hh}{1+\hh}\right)^i+\left(-\frac{\hh}{1+\hh}\right)^{i+1}\right) c_*(Eu_X)\\
&=(-1)^d\int_{\bP^n}\frac{1}{1+\hh}\left(-\frac{\hh}{1+\hh}\right)^i c_*(Eu_X)\\
&=(-1)^d\int_{\bP^n}\frac{1}{1+\hh}\left(-\frac{\hh}{1+\hh}\right)^i \left(\sum_{j\geq 0}a_j \hh^{n-j}\right).
\end{align*}
Thus,
\begin{align*}
\sum_{0\leq i \leq d} s_i \cdot t^{n-i}
&=\sum_{i\geq 0}(-1)^d\int_{\bP^n}\frac{1}{1+\hh}\left(-\frac{\hh}{1+\hh}\right)^i \left(\sum_{j\geq 0}a_j \hh^{n-j}\right)\cdot t^{n-i}\\
&=(-1)^d\int_{\bP^n}\sum_{j\geq 0}a_j\hh^{n-j}\frac{1}{1+\hh} \sum_{i\geq 0}\left(-\frac{\hh\cdot  t^{-1}}{1+\hh}\right)^i\cdot t^{n}\\
&=(-1)^d\int_{\bP^n}\sum_{j\geq 0}a_j\hh^{n-j}\frac{1}{1+\hh} \left(1+\frac{\hh\cdot t^{-1}}{1+\hh}\right)^{-1}\cdot t^{n}\\
&=(-1)^d\int_{\bP^n}\sum_{j\geq 0}a_j\hh^{n-j}\frac{1}{1+\hh(1+ t^{-1})}\cdot t^{n}\\
&=(-1)^d\int_{\bP^n}\sum_{i, j\geq 0}a_j\hh^{n-j}(-\hh)^i(1+t^{-1})^i\cdot t^{n}\\
&=(-1)^d \sum_{j\geq 0}a_j(-1)^j(1+t^{-1})^j\cdot t^{n}\\
&=(-1)^d \sum_{j\geq 0}a_j(-1)^j(1+t)^j\cdot t^{n-j}.
\end{align*}
On the other hand,  formula \eqref{id1} of Theorem \ref{thm_main} shows that the last term in the above sequence of equalities is exactly $\sum_{0\leq i\leq d}b_i \cdot t^{n-i}$. Hence $b_i=s_i$, for all $0\leq i\leq d$.
\end{proof}

\section{LO bidegrees and projective polar degrees}\label{sec_compare}
Let $X$ be an affine variety in $\bC^n$ and let $\overline{X}$ be its closure in $\bP^n$. As before, we use $T^*_X\bC^n\subset T^*\bC^n=\bC^n\times \bC^n$ to denote the {affine} conormal variety of $X$, and following the notations of \cite{Sturmfels}, we use $N_{\overline{X}}\subset \bP^n\times (\bP^n)^\vee$ to denote the projective conormal variety of $\overline{X}$. In this section, we compare the two conormal varieties and their bidegrees. 

First, let us review the definitions of the affine and projective conormal varieties. For simplicity, we start with the case when $\overline{X}$, and hence $X$, is smooth. In this case, 
\[
T^*_X\bC^n=\left\{(\bfx, \bfu)\in T^*\bC^n=\bC^n_x\times \bC^n_u\mid \bfx\in X \text{ and } \bfu|_{T_\bfx X}=0\right\}.
\]
Here $\bfu=(u_1, \ldots, u_n)$ corresponds to the parallel 1-form $\sum_{1\leq i\leq n}u_i dx_i$ on $\bC^n$, and $\bfu|_{T_{\bfx} X}=0$ means that $\bfx$ is a critical point of the linear function $\sum_{1\leq i\leq n}u_i x_i$. Equivalently, $\bfu|_{T_{\bfx} X}=0$ if and only if a level set of $\sum_{1\leq i\leq n}u_i x_i$ is tangent to $X$ at $\bfx$. On the other hand, following \cite{RS2013}, the \emph{projective conormal variety} is the $(n-1)$-dimensional subvariety of $\bP^n\times (\bP^n)^\vee$ defined by
\[
N_{\overline{X}}=\left\{(\bfp, H)\in \bP^n\times (\bP^n)^\vee\mid \bfp\in \overline{X} \text{ and } H \text{ is tangent to $\overline{X}$ at }\bfp\right\},
\]
where the dual projective space $(\bP^n)^\vee$ parametrizes hyperplanes in $\bP^n$. In general, when $X$ or $\overline{X}$ is singular, we use the above formulas to define the conormal varieties along the smooth locus, $T^*_{X_{\reg}}\bC^n$ and $N_{\overline{X}_{\reg}}$. Then we let $T^*_X\bC^n$ and $N_{\overline{X}}$ be their closures in $T^*\bC^n$ and $\bP^n\times (\bP^n)^\vee$, respectively. 

Let $H_\infty\in (\bP^n)^\vee$ denote the hyperplane at infinity in $\bP^n$, and let $\pi_\infty: (\bP^n)^\vee \dashrightarrow\bP^{n-1}$ be the rational map given by  projecting from $H_\infty$. Since the $\bC^*$-action on $\bC^n_u$ by scalar multiplication preserves the subvariety $T^*_X\bC^n\subset \bC^n_x\times \bC^n_u$, we can take the fiberwise projectivization $\bP(T^*_X\bC^n)\subset \bC^n_x\times \bP^{n-1}$, and denote its closure in $\bP^n\times \bP^{n-1}$ by $\overline{\bP(T^*_X\bC^n)}$. Then the affine and projective conormal varieties are related as follows. 
\begin{proposition}\label{prop_pi}
Assume that $X$ is not contained in any proper affine subspace, that is, $\overline{X}$ is not contained in a hyperplane. Under the above notations, the rational map $\id\times \pi_\infty: \bP^n\times (\bP^{n})^\vee\dashrightarrow \bP^n\times \bP^{n-1}$ restricts to a birational map between $N_{\overline{X}}$ and $\overline{\bP(T^*_X\bC^n)}$. Hence, we have an equality of subvarieties of $\bP^n\times \bP^{n-1}$,
\begin{equation}\label{eq_proj}
\id\times \pi_\infty(N_{\overline{X}})=\overline{\bP(T^*_X\bC^n)},
\end{equation}
where we regard the left-hand side as the closure of $\id\times \pi_\infty(N_{\overline{X}}\setminus \bP^n\times \{H_\infty\})$.
\end{proposition}
\begin{proof}
Since both $N_{\overline{X}}$ and $\bP(T^*_X\bC^n)$ are irreducible, it suffices to show that $\id\times \pi_\infty$ induces a bijection between general points in $N_{\overline{X}}$ and $\bP(T^*_X\bC^n)$. 
In fact, fix a general point $(\bfx, H)\in N_{\overline{X}}$, where $\bfx\in X_{\reg}$ and ${H}$ is tangent to $X$ at $\bfx$. Let the defining equation of $H$ be $u_0p_0+\cdots+u_np_n=0$, where $p_i$ are the homogeneous coordinates of $\bP^n$. The restriction of $H$ to the affine chart $p_0\neq 0$ is the level set $\{u_1x_1+\cdots+u_nx_n=-u_0\}$ of the linear function $l_\bfu\coloneqq u_1x_1+\cdots+u_nx_n$, where $x_i=p_i/p_0$ are the affine coordinates. 
The projection $\id\times \pi_\infty$ forgets the value of $u_0$ and only remembers the linear function $l_\bfu$ (up to scalar). Given the point $\bfx$ and the linear function $l_\bfu$, there is a unique level set of $l_\bfu$ containing $\bfx$. 
Thus, the restriction of $\id\times \pi_\infty$ to $N_{\overline{X}}$ is generically injective. In other words, $\id\times \pi_\infty$ induces a birational equivalence between $N_{\overline{X}}$ and its image. 

Now, we only need to prove the equality \eqref{eq_proj}. By the earlier discussions, putting $\bfu=(u_1, \ldots, u_n)$, then $(\bfx, \bfu)$ defines a point in $\bP(T^*_X\bC^n)$. Conversely, a general point $(\bfx, \bfu)$ of $\bP(T^*_X\bC^n)$ corresponds to a linear function $l_\bfu=u_1x_1+\cdots+u_nx_n$ (up to scalar) and a critical point $\bfx$ of $l_\bfu|_X$. Let $H$ be the {projective closure of the} level set of $l_\bfu$ containing $\bfx$. Then $\id\times \pi_\infty(\bfx, H)=(\bfx, \bfu)$. Thus, equality \eqref{eq_proj} follows. 
\end{proof}

Using the above result, we will derive relations between the LO bidegrees of $X$ and the polar degrees of $\overline{X}$. First, let us recall the definitions. As in the introduction, the LO bidegrees $b_i(X)$ (or simply $b_i$) are the bidegrees of the closure of the {affine} conormal variety $T^*_X\bC^n$ in $\bP^n\times \bP^n$. More precisely, they are defined by the following formula
\[
[\overline{T^*_X\bC^n}]=b_0 [\bP^0\times \bP^n]+b_1[\bP^{1}\times \bP^{n-1}]+\cdots +b_d[\bP^{d}\times \bP^{n-d}]\in A_*(\bP^n\times \bP^n),
\]
where $d=\dim X$.
Similarly, the \emph{polar degrees} $\delta_i(\overline{X})$ (or simply $\delta_i$) of $\overline{X}$ are the bidegrees of the projective conormal variety $N_{\overline{X}}\subset \bP^n\times (\bP^n)^\vee$. More precisely, they are defined by (see e.g., \cite[Section 2]{Sturmfels})
\[
[N_{\overline{X}}]=\delta_1 [\bP^{0}\times \bP^{n-1}]+\cdots +\delta_{d+1}[\bP^{d}\times \bP^{n-d-1}].
\]

\begin{proposition}\label{prop_iff}
The bidegrees of $T^*_X\bC^n\subset \bC^n_x\times \bC^n_u$ and the bidegrees of $N_{\overline{X}}\subset \bP^n\times (\bP^n)^\vee$ coincide in the sense that
\begin{equation}\label{eq_delta}
b_i(X)=\delta_{i+1}(\overline{X}), \quad\text{for }0\leq i\leq d
\end{equation}
if and only if the hyperplane at infinity $H_\infty$ is not a point in the dual variety $\overline{X}^\vee\subset (\bP^n)^\vee$.
\end{proposition}
\begin{proof}
Fixing $i$, let $L^{n-i}\subset \bP^n$ be a general linear subspace of dimension $n-i$, and let $L^{i}\subset \bP^{n-1}$ be a general linear subspace of dimension $i$. By Bertini's theorem, $L^{n-i}\times L^{i}$ intersects $\bP(T^*_X\bC^n)$ in $\bP^n\times \bP^{n-1}$ transversally, and the intersection consists of $b_i(X)$ points. Now, we assume that $H_\infty$ is not in $\overline{X}^\vee$, that is, $N_{\overline{X}}\cap (\bP^n\times \{H_\infty\})=\emptyset$. 
Let $M^{i+1}\subset (\bP^n)^\vee$ be a general linear space of dimension $i+1$ passing through the point $H_\infty$. Then $L^{n-i}\times M^{i+1}$ is cut out by $i$ general hyperplanes in $\bP^n$ and $n-i-1$ general hyperplanes in $(\bP^n)^\vee$ passing through $H_\infty$. 
Since $N_{\overline{X}}\cap (\bP^n\times \{H_\infty\})=\emptyset$, by Bertini's theorem, $L^{n-i}\times M^{i+1}$ intersects $N_{\overline{X}}$ transversally  at $\delta_{i+1}(\overline{X})$ points. 
By Proposition \ref{prop_pi}, if $M^{i+1}$ is the cone over $L^{i}$ with vertex $H_\infty$, then $\id\times \pi_\infty$ induces a bijection between $N_{\overline{X}}\cap (L^{n-i}\times M^{i+1})$ and $ \bP(T^*_X\bC^n)\cap(L^{n-i}\times L^{i})$. Hence $b_i(X)=\delta_{i+1}(\overline{X})$. 

Next, we assume that $H_\infty$ is in $\overline{X}^\vee$. Let $e$ be the codimension of $\overline{X}^\vee$ in $(\bP^n)^\vee$. We claim that $b_{e-1}(X)<\delta_{e}(\overline{X})$. Let $L^{n-e+1} \in \bP^{n}$ be a general linear space of dimension $n-e+1$, let $L^{e}\subset (\bP^n)^\vee$ be a general linear space of dimension $e$, and let $M^{e}\subset (\bP^n)^\vee$ be a general linear space of dimension $e$ containing $H_\infty$. Denote by $p_2: \bP^n\times (\bP^{n})^\vee\to (\bP^{n})^\vee$ the second projection and, by abusing notations, we also use $p_2$ to denote any of its restrictions. Since the dual variety $\overline{X}^\vee$ has codimension $e$, any fiber of the map
\[
p_2: N_{\overline{X}}\to \overline{X}^\vee
\]
has dimension at least $e-1$, and a general fiber has dimension exactly $e-1$. Thus, the restriction 
\[
p_2: N_{\overline{X}}\cap (L^{n-e+1}\times (\bP^{n})^\vee)\to \overline{X}^\vee
\]
is surjective and generically finite, whose degree we denote by $h$. Then $\overline{X}^\vee$ intersects $L^e$ transversally at ${\delta_e(\overline{X})}/{h}$ points.  By Bertini's theorem and Proposition \ref{prop_pi}, away from $H_\infty$, $\overline{X}^\vee$ intersects $M^e$ transversally at $b_{e-1}(X)/h$ points. By assumption, $H_\infty$ is contained in the intersection of $\overline{X}^\vee$ and $M^e$. Moreover, the intersection multiplicity of $\overline{X}^\vee\cdot M^e$ at $H^\infty$ is positive (see \cite[Proposition 7.1]{Ful} or \cite[Section V.3, Theorem~1]{Serre}). Since the global intersection numbers satisfy $\overline{X}^\vee\cdot L^e=\overline{X}^\vee\cdot M^e$, the above arguments imply that $b_{e-1}(X)/h<\delta_{e}(\overline{X})/h$, that is, $b_{e-1}(X)<\delta_{e}(\overline{X})$. 
\end{proof}

Combining Theorem \ref{thm_bs} and Proposition \ref{prop_iff}, we obtain the following generalization of \cite[Theorem 13]{wasserstein} (see also \cite[Proposition 2.9]{Sturmfels}) to singular varieties. 
\begin{corollary}\label{cor_si} Let $X\subset \bC^n$ be an affine variety, with projective closure 
$\overline{X} \subset \bP^n$. 
Assume that the hyperplane at infinity $H_\infty$ is not contained in $\overline{X}^\vee$. Then the sectional LO degrees of $X$ coincide with the polar degrees of $\overline{X}$, that is, $s_i(X)=\delta_{i+1}(\overline{X})$ for all $0\leq i\leq \dim X$. 
\end{corollary}
\begin{remark}
If the affine variety $X\subset \bC^n$ is defined by homogeneous polynomials, i.e., $X$ is the cone of a projective variety, then its closure intersects the hyperplane at infinity $H_\infty$ transversally. In this case, $H_\infty$ is not contained in $\overline{X}^\vee$, and hence we have $\delta_{i+1}(\overline{X})=s_i(X)=b_i(X)$, for all  $0\leq i\leq \dim X$.
For example, if $X\subset \CC^9$ 
is defined by the vanishing of the determinant of the matrix 
$\begin{bsmallmatrix} 
x_0&x_1&x_2\\
x_3&x_4&x_5\\
x_6&x_7&x_8\\
\end{bsmallmatrix}$ 
then the 
LO bidegrees of $X$ 
and the polar degrees of $\overline{X}$ are given by
\begin{align*}
[\overline{T^*_X\bC^9}]&=
6 [\bP^0\times \bP^9]+   12 [\bP^1\times \bP^8]+   12 [\bP^2\times \bP^7]+   6 [\bP^3\times \bP^6]+   3 [\bP^4\times \bP^5],\\
[N_{\overline{X}}] & =
6 [\bP^0\times \bP^8]+   12 [\bP^1\times \bP^7]+   12 [\bP^2\times \bP^6]+   6 [\bP^3\times \bP^5]+   3 [\bP^4\times \bP^4].
\end{align*}
\end{remark}

\medskip

The following examples illustrate that, when $H_\infty\in \overline{X}^\vee$, the two sets of bidegrees considered above are different. 
\begin{ex}
Let $X$ in $\CC^3$ be the curve $\VV(x^2+y^2+z^2-1,y-x^2)$. 
The curve $X$ and its projective closure $\overline X=\VV(x^2+y^2+z^2-u^2,yu-x^2)$ are smooth. 
The LO bidegrees of $X$ 
and the polar degrees of $\overline{X}$ are given by
\begin{align*}
[\overline{T^*_X\bC^3}]&=6 [\bP^0\times \bP^3]+4[\bP^{1}\times \bP^{2}],\\
[N_{\overline{X}}] & =8 [\bP^{0}\times \bP^{2}] +4[\bP^{1}\times \bP^{1}].
\end{align*}
In this case, $\overline{X}^\vee$ has codimension 1, and as predicted by Proposition \ref{prop_iff}, $b_0=6<\delta_1=8$.
Note that this example satisfies the assumption that the real algebraic curve obtained by intersecting $X$ and $\mathbb{R}^3$ is compact (compare with \cite[]{Wasserstein}) 
\end{ex}

\begin{ex}
If $X\subset \CC^4$ is the hypersurface $\VV(x_1^2x_2-x_3x_4)$
then its projective closure is 
$\overline X=\VV(x_1^2x_2-x_0 x_3 x_4 )$. 
The dual variety $\overline{X}^\vee$ contains the point $[0:0:0:0:1]$ and is defined by the binomial $y_1^2y_2+4y_0y_3y_4$.
The LO bidegrees of $X$ 
and the polar degrees of $\overline{X}$ are very different: 
\begin{align*}
[\overline{T^*_X\bC^n}]&=1 [\bP^0\times \bP^4]+4[\bP^{1}\times \bP^{3}]   +5[\bP^{2}\times \bP^{2}]   +3[\bP^{3}\times \bP^{1}]   ,\\
[N_{\overline{X}}] &=3 [\bP^0\times \bP^3]+6[\bP^{1}\times \bP^{2}]   +6[\bP^{2}\times \bP^{1}]   +3[\bP^{3}\times \bP^{0}].
\end{align*}
\end{ex}

\section{Applied Algebraic Geometric Context}\label{sec:applied}
The \emph{algebraic degree of an optimization problem} is a well studied topic in 
applied algebraic geometry. 
It appears in numerous fields including statistics~\cite{MR2230921,MR2988436,RW}, semidefinite programming~\cite{MR2496496,MR2546336,RS2013}, computer vision~\cite{HARTLEY1997146,MRW}, physics~\cite{MR4103774,ST2021}, and 
 polynomial optimization~\cite{MR2507133}. 
A recent useful machine learning application of optimizing linear functions on varieties appears in Wasserstein optimization~\cite{wasserstein} 
and, more generally, in computing the distance to a variety with respect to any polyhedral norm~\cite[Equation 3.2]{wasserstein}.
In many of these applications, the computational results rely on the notion of genericity and counting intersection points.  In this section, we bring our results into this realm to provide a bijection between critical points of a linear objective function  on $X$ restricted to a linear space with a set of points in the {affine} conormal~variety.

We fix an affine variety $X$ in $\CC^n$. 
Recall that the cotangent space $T^*\CC^n\simeq \CC_x^n\times \CC_u^n$,
where $x$ and $u$  denote the coordinates of vector and covector components of $T^*\CC^n$.
The conormal variety $T^*_X\bC^n$ inside of $\CC_x^n\times \CC_u^n$ is 
of dimension $n$. 
Let $L$  be an affine subspace of codimension $i$.
Throughout this section, for $\bfu=(u_1,\dots,u_n)$ in $\CC^n_u$, 
we define the linear function $h_\bfu:\CC^n\to \CC, (x_1,\dots,x_n)\mapsto u_1x_1+\cdots u_nx_n$.  
We let $L^\perp_\bfu\subset \CC_u^n$ denote the affine subspace orthogonal to $L$ containing $\bfu$. Then $L\times L^\perp_\bfu\subset \CC_x^n\times \CC^n_u$ is $n$-dimensional.

\begin{proposition}\label{prop:LM}
Let $\bfu$ be a generic point in $\CN_u$ and 
$L\subset \CC^n_x$ be a generic affine subspace of codimension $i$. 
Then, $\bfp$ is a critical point of $h_\bfu$ restricted to $(X\cap L)_{\reg}$ 
if and only if there exists a unique $\bfy\in \CC^n$ such that 
$(\bfp,\bfy) \in T_X^* \CC^n \cap (L\times L^\perp_\bfu)$. 
\end{proposition}

\begin{proof}
Since $L$ is generic, 
the sectional LO degree $s_i(X)$ is the number of critical points of $h_\bfu$ on $(X\cap L)_{\reg}$. 
Denote this set of critical points by $W_\bfu$.
For each $\bfp\in W_\bfu$, we know $\bfu$ is in the row span of the Jacobian of generators of the ideal of $X\cap L$ evaluated at $\bfp$ because $\bfp$ is a critical point of the linear function $h_\bfu$ restricted to $(X\cap L)_{\reg}$.
In other words,  there exist $\bfy\in \CC^n$ and $\bfz\in \CC^n$ such that $\bfu=\bfz+\bfy$ with
\begin{enumerate}
\item $\bfz=\bfy-\bfu$ is in the row span of the Jacobian of generators of the ideal of $L$,
\item $\bfy$ in the row span of the Jacobian of generators of the ideal of $X$ evaluated at $p$.
\end{enumerate}
Recall $L^\perp_\bfu$ is the orthogonal complement of $L$ translated to pass through the point $u$. 
Therefore $\bfy\in L^\perp_\bfu$.
So for $\bfp\in W_\bfu$, we have $(\bfp,\bfy)\in T^*_X\CC^n\cap (L\times L^\perp_\bfu)$. 
The other implication is immediate. 
\end{proof}

\begin{corollary}\label{cor:generic}
For $\bfu$ and $L$ as in Proposition~\ref{prop:LM},
there are $b_i(X)$ points of intersection in $T_X^* \CC^n \cap (L\times L^\perp_\bfu)$.
\end{corollary}

\begin{proof}
Let $M$ be a generic affine linear space of codimension $n-i$. 
Then $T_X^* \CC^n \cap (L\times M)$ consists of $b_i$ points of intersection by definition of LO bidegree.  
Since $L$ is generic and $L^\perp_\bfu$ is a generic translate, 
we have $L\times L^\perp_\bfu$ is a generic translate and $T^*_X\CC^n \cap (L\times L^\perp_\bfu)$ consists of finitely many points (by Bertini's theorem). 
Since $L^\perp_\bfu$ is of codimension $n-i$, the number of points in  $T^*_X\CC^n \cap (L\times L^\perp_\bfu)$ is at most $b_i$. 
Thus, it suffices to show the cardinality of $T^*_X\CC^n \cap (L\times L^\perp_\bfu)$ is at least $b_i$. This follows from Proposition~\ref{prop:LM}.
\end{proof}

\begin{remark}Corollary~\ref{cor:generic} is in stark contrast to the maximum likelihood degree case~\cite{MRWW} where the ML bidegree $b_i$ is usually a strict upper bound on the ML sectional degree $s_i$. 
\end{remark}

\begin{ex}[Illustrative]
Let $X$ be the sphere in $\CC^3$ defined by $x_1^2+x_2^2+x_3^2=100$. 
The LO bidegrees and sectional LO-degrees of $X$ are 
$(b_0, b_1, b_2)=(s_0,s_1,s_2)=(2,2,2)$. 
With this setup, our interest is in $b_1(X)$ and $s_1(X)$.
We let $\bfu=(10,5,17)$ and $L=\VV(x_3-6)$ so that $L^\perp_\bfu=\VV(y_1-10,y_2-5)$.
Then, $T^*_X\CC^n \cap (L\times L^\perp_\bfu)$ has $b_1=2$ points. 
To compute $s_1(X)$, we can find the set of two critical points of $h_\bfu$ restricted to $(X\cap L)_{\reg}$ to be
\begin{equation}\label{eq:twopoints}
\{ (2\alpha,\alpha, 6 )\in \CC^3 :  5\alpha^2=64 \}.
\end{equation}
 
 From \eqref{eq:twopoints},
 we can recover $T^*_X\CC^n \cap (L\times L^\perp_\bfu)$
 by following the proof in Proposition \ref{prop:LM}.
We have
the Jacobian of $\{ x_1^2+x_2^2+x_3^2-100 , x_3-6\}$ evaluated at 
$\bfp=(2\alpha,\alpha, 6 )$ is 
\[
\begin{bmatrix}\label{ex:sphere}
4\alpha &2\alpha &12\\
0 & 0 & 1
\end{bmatrix}.
\]
We see how $\bfu$ is a linear combinators of the rows of the evaluated Jacobian:
\[
\bfu = (10,5,17)= \frac{10\cdot 5\alpha}{4\cdot 64}\cdot (4\alpha, 2\alpha, 12) + (-\frac{75}{32}\alpha+17) \cdot (0,0,1).
\]
We take 
$\bfy =  (10,5, \frac{75}{32}\alpha) $
and so $\bfy-\bfu=(0,0,\frac{75}{32}\alpha-17)$ is in the row span of $(0,0,1)$. 
Thus, 
\[
T^*_X\CC^n \cap (L\times L^\perp_\bfu)=
\left\{
\left( 2\alpha,\alpha, 6   ,\,  10,5, \frac{75}{32}\alpha\right)\in \CC^3\times \CC^3 :5\alpha^2=64
\right\}.
\]
\end{ex}

\begin{remark}
In Example~\ref{ex:sphere}, we chose $L$ to be a general coordinate hyperplane for illustrative purposes.  
This is not sufficiently generic for every example. For instance, 
if we let $f=1+x_1+x_2^2+x_3^3$ instead, then the LO bidegrees of $V(f)$ are the sequence $(2,4,3)$. 
However,  $T^*_X\CC^n \cap (L\times L^\perp_\bfu)$ has only one point:
$(-3473/16, 1/4, 6,\, 10, 5, 1080)$, and so  $L\times L^\perp_\bfu$ does not intersect the {affine} conormal variety at $b_1$ points.
\end{remark}

\bibliographystyle{abbrv}

\end{document}